\documentclass[a4paper, reqno, 11pt]{amsart}

\usepackage[usenames,dvipsnames]{color}
\usepackage{amsthm,amsfonts,amssymb,amsmath,amsxtra}
\usepackage{tikz}
\usepackage{tkz-euclide}
\usepackage[all]{xy}
\SelectTips{cm}{}
\usepackage{xr-hyper}
\usepackage[
   citecolor=Black,
   linkcolor=Red,
   urlcolor=Blue]{hyperref}
\usepackage{verbatim}

\usepackage[margin=1.25in]{geometry}
\usepackage{mathrsfs}

\RequirePackage{xspace}
\RequirePackage{etoolbox}
\RequirePackage{varwidth}
\RequirePackage{enumitem}
\RequirePackage{tensor}
\RequirePackage{mathtools}
\RequirePackage{longtable}
\RequirePackage{multirow}

\usepackage{environ}

\def\JJ{\mathbb J}

\def\ge{\geqslant}
\def\le{\leqslant}

\def\s{\sigma}
\def\t{\tau}

\def\<{\langle}
\def\>{\rangle}

\newcommand{\sF}{\ensuremath{\mathscr{F}}\xspace}

\newcommand{\fka}{\ensuremath{\mathfrak{a}}\xspace}
\newcommand{\fkb}{\ensuremath{\mathfrak{b}}\xspace}

\newcommand{\fkg}{\ensuremath{\mathfrak{g}}\xspace}
\newcommand{\fkh}{\ensuremath{\mathfrak{h}}\xspace}

\newcommand{\fkj}{\ensuremath{\mathfrak{j}}\xspace}

\newcommand{\fkx}{\ensuremath{\mathfrak{x}}\xspace}
\newcommand{\fky}{\ensuremath{\mathfrak{y}}\xspace}

\def\brF{\breve F}

\def\brI{\breve \CI}
\def\brK{\breve \CK}
\def\brP{\breve \CP}

\def\brsA{\breve{\mathscr A}}

\def\tSS{\tilde{\mathbb S}}

\newcommand{\BF}{\ensuremath{\mathbb {F}}\xspace}
\newcommand{{\BG}}{\ensuremath{\mathbb {G}}\xspace}

\newcommand{{\BK}}{\ensuremath{\mathbb {K}}\xspace}

\newcommand{\BS}{\ensuremath{\mathbb {S}}\xspace}

\newcommand{\CI}{\ensuremath{\mathcal {I}}\xspace}

\newcommand{\CK}{\ensuremath{\mathcal {K}}\xspace}

\newcommand{\CP}{\ensuremath{\mathcal {P}}\xspace}

\newcommand{\CR}{\ensuremath{\mathcal {R}}\xspace}

\DeclareMathOperator{\dist}{dist}

\DeclareMathOperator{\Adm}{Adm}

\newcommand{\supps}{\mathop{\rm supp}\nolimits_{\s}}
\newcommand{\EOcox}{\mathop{\rm EO}\nolimits^K_{\s, {\rm cox}}}

\DeclareMathOperator{\Gal}{Gal}

\newcommand{\id}{\ensuremath{\mathrm{id}}\xspace}

\newcommand{\inv}{{\mathrm{inv}}}

\DeclareMathOperator{\Ker}{Ker}

\DeclareMathOperator{\proj}{proj}

\def\tW{\tilde W}

\DeclareMathOperator{\supp}{supp}
\newcommand{\Flag}{\mathscr Flag}

\newtheorem{theorem}{Theorem}

\newtheorem{proposition}[theorem]{Proposition}
\newtheorem{lemma}[theorem]{Lemma}

\newtheorem{corollary}[theorem]{Corollary}

\theoremstyle{definition}
\newtheorem{definition}[theorem]{Definition}
\newtheorem{example}[theorem]{Example}

\newtheorem{remark}[theorem]{Remark}

\numberwithin{equation}{section}
\numberwithin{theorem}{section}

\setitemize[0]{leftmargin=*,itemsep=\the\smallskipamount}
\setenumerate[0]{leftmargin=*,itemsep=\the\smallskipamount}

\renewcommand{\to}{%
   \ifbool{@display}{\longrightarrow}{\rightarrow}%
   }
\let\shortmapsto\mapsto
\renewcommand{\mapsto}{%
   \ifbool{@display}{\longmapsto}{\shortmapsto}%
   }
\newcommand{\isoarrow}{%
   \ifbool{@display}{\overset{\sim}{\longrightarrow}}{\xrightarrow\sim}%
   }

\newcommand{\CVstrat}{$\JJ$-stratification}
\newcommand{\BTstrat}{Bruhat-Tits stratification}

\begin{document}
\author[U.~G\"{o}rtz]{Ulrich G\"{o}rtz}
\address{Ulrich G\"{o}rtz\\Institut f\"ur Experimentelle Mathematik\\Universit\"at Duisburg-Essen\\45117 Essen\\Germany}
\email{ulrich.goertz@uni-due.de}
\thanks{The author was partially supported by DFG Transregio-Sonderforschungsbereich 45.}

\title{Stratifications of affine Deligne-Lusztig varieties}
\begin{abstract}
Affine Deligne-Lusztig varieties are analogues of Deligne-Lusztig varieties in the context of affine flag varieties and affine Grassmannians. They are closely related to moduli spaces of $p$-divisible groups in positive characteristic, and thus to arithmetic properties of Shimura varieties.

We compare stratifications of affine Deligne-Lusztig varieties attached to a basic element $b$. In particular, we show that the stratification defined by Chen and Viehmann using the relative position to elements of the group $\JJ_b$, the $\sigma$-centralizer of $b$, coincides with the Bruhat-Tits stratification in all cases of Coxeter type, as defined by X.~He and the author.
\end{abstract}

\maketitle

\setcounter{section}{-1}
\section{Introduction}
Affine Deligne-Lusztig varieties are closely related to moduli spaces of $p$-divisible groups (so-called Rapoport-Zink spaces), and hence to the reduction of Shimura varieties and their arithmetic properties. They are also interesting in their own: Similarly as with classical Deligne-Lusztig varieties, their cohomology gives rise to representations (of certain $p$-adic groups). For instance, one can hope that this gives an approach to the local Langlands correspondence, compare the papers by Chan~\cite{Chan} and Ivanov~\cite{Ivanov1}, \cite{Ivanov2}.

In general, the geometric structure of affine Deligne-Lusztig varieties is hard to understand. Even foundational properties such as the dimension and the sets of connected and irreducible components are not yet fully understood in general. See the survey papers by Rapoport~\cite{rapoport:guide} and He~\cite{He-CDM} and the references given there.

A general approach to understanding the geometry and cohomology of an affine Deligne-Lusztig variety is to stratify it as a union of better understood varieties. In \cite{CV}, Chen and Viehmann define a stratification (which we below call the \CVstrat) of affine Deligne-Lusztig varieties in a very general setting, using the relative position to elements of the group $\JJ$, the $\s$-centralizer of the underlying element $b$. While in their paper they consider the case where the level structure is hyperspecial (and so in particular the underlying algebraic group $G$ is assumed to be unramified), it is clear that their underlying definition applies in the general case, as well. However, it is extremely difficult to give a good description of this stratification: Neither the index set of the stratification nor the geometric structure of the strata are easy to understand. Also, the \CVstrat{} is a stratification only in the loose sense that it is a decomposition into disjoint locally closed subsets; in general, the closure of a stratum is not a union of strata. Nevertheless, Chen and Viehmann are able to give a convincing justification of their definition by showing that in several specific cases their stratification coincides with stratifications which were studied previously and which have nice properties.

One of these is the so-called Bruhat-Tits (BT) stratification, a stratification indexed in terms of the Bruhat-Tits building of the algebraic group $\JJ$ (whence the name), whose strata are classical Deligne-Lusztig varieties. This stratification exists only in particular cases, called of \emph{Coxeter type} in \cite{GH}, where the situation was discussed in detail from the point of view of affine Deligne-Lusztig varieties; see also \cite{GHN2}. Even though this is quite a special situation, it includes many important cases and this phenomenon had been studied extensively in the context of Rapoport-Zink spaces and Shimura varieties (see e.~g.~\cite{Kaiser}, \cite{Vollaard-Wedhorn}, \cite{Rapoport-Terstiege-Wilson}) where it plays an important role, for instance in the Kudla-Rapoport program.

The purpose of this paper is to provide the technical result ensuring that the $\JJ$-strata are locally closed subsets of $\Flag_K$ without the restriction to hyperspecial level structure, and to prove that the $\JJ$-stratification coincides with the BT stratification in all cases of Coxeter type, as conjectured in \cite{CV}.

To describe the results in more detail, we fix some notation.
Let $F$ be a non-archimedean local field, i.e., a finite extension of $\mathbb Q_p$, the field of $p$-adic numbers for a prime $p$, or of the form $\mathbb F_q((t))$ where $\mathbb F_q$ denotes the finite field with $q$ elements. Let $\brF$ be the completion of the maximal unramified extension of $F$, and let $\sigma$ be the Frobenius automorphism of $\brF$ with fixed field $F$.

Let $G$ be a connected semisimple group of adjoint type over $F$. We denote by $\tW$ the extended affine Weyl group over $\brF$, and by $\tSS\subset \tW$ the set of simple affine reflections. Let $K\subset \tSS$ (a ``level structure''), and let $\brK\subset G(\brF)$ be the corresponding parahoric subgroup (see Section~\ref{sec:notation} for more details).

For $b\in G(\brF)$ and $w\in \tW$ we have the \emph{affine Deligne-Lusztig variety}
\[
X_w(b) = \{ g\in G(\brF)/\brK;\ g^{-1}b\sigma(g) \in \brK w\brK \},
\]
a locally closed subscheme of the ind-scheme $G(\brF)/\brK$ (in the case of mixed characteristic, we view $G(\brF)/\brK$ and $X_w(b)$ as ind-perfect schemes in the sense of Zhu~\cite{Zhu} and Bhatt and Scholze~\cite{Bhatt-Scholze}). We also have the variant $X(\mu, b)_K$ for $\mu$ a translation, the union of the above $X_w(b)$ for $w$ in the $\mu$-admissible set. See Section~\ref{sec:adlv}.

To state the definition of the $\JJ$-stratification, denote by
\[
\JJ :=\{ g\in G(\brF);\ g^{-1}b\s(g) = b \}
\]
the $\s$-centralizer of $b$. (Since $b$ is usually fixed in the discussion, we mostly omit it from the notation.) By definition, two points $g, g'\in \Flag_K$ lie in the same $\JJ$-stratum, if and only if for all $j\in \JJ$, $\inv_K(j, g) = \inv_K(j, g')$. Here $\inv_K$ denotes the relative position map for level $\brK$ (see Def.~\ref{def-relative-position}). For the definition of the Bruhat-Tits stratification, we refer to Section~\ref{sec:three} and to~\cite{GH}.

The main theorem in this paper is

\begin{theorem}{\rm (\cite{CV}, Conj.~4.2; Theorem~\ref{main-thm}, Remark~\ref{non-basic-case})}
Let $(G, \mu, K)$ be of Coxeter type. Then the Bruhat-Tits stratification on $X(\mu, b)_K$ coincides with the $\JJ$-stratification.
\end{theorem}

See Section~\ref{sec:coxeter-setup} for details on the setup, including the (important) choices of the particular elements $b$. To prove that the Bruhat-Tits stratification is a refinement of the $\JJ$-stratification we use a result of Lusztig (Prop.~\ref{prop-coxeter-dlv}) saying that a classical Deligne-Lusztig variety attached to a twisted Coxeter element is contained in the open cell of the ambient flag variety. The converse is more difficult and relies, in addition, on combinatorial arguments on the affine root system and the building of $G$; one of the ingredients is the theory of acute cones developed by Haines and Ng\^{o}, \cite{HN}.

Before we come to the main theorem, we generalize the following finiteness property of the \CVstrat{} which was proved for hyperspecial level in~\cite{CV}, Prop.~2.6.

\begin{theorem}{\rm (Theorem~\ref{thm-finiteness})}
Let $b\in G(\brF)$, let $\JJ_b$ be its $\s$-centralizer, and let $\JJ=\JJ_b(F)$.
Fix a parahoric level $K\subset \tSS$ and a quasi-compact subscheme $S\subset \Flag_K$ in the partial flag variety for $K$.

There exists a finite subset $J'\subset \JJ$ such that for all $g, g'\in S$ with
\[
\inv(j, g) = \inv(j, g')\ \text{for all } j\in J',
\]
we have
\[
\inv(j, g) = \inv(j, g')\ \text{for all } j\in \JJ.
\]
\end{theorem}

The result of the theorem ensures that the $\JJ$-stratification has certain reasonable finiteness properties; in particular, the $\JJ$-strata are locally closed. See~\cite{CV} Cor.~2.10, Prop.~2.11.

\emph{Outline of the paper.} In Section 1, we recall Lusztig's result on classical Deligne-Lusztig varieties mentioned above, give the details for the setup concerning affine Deligne-Lusztig varieties and recall some properties of the Bruhat-Tits building that we are going to use. In Section 2, we discuss the definition and general properties of the $\JJ$-stratification, including the finiteness property stated in the introduction. Finally, in Section 3 we recall the notion of data of Coxeter type and the Bruhat-Tits stratification, and prove that it coincides with the $\JJ$-stratification.


\section{Preliminaries}

\subsection{Deligne-Lusztig varieties}\label{sec:lusztig}

Let $k_0$ be a finite field, and let $k$ be an algebraic closure of $k_0$. We consider a connected reductive group $G_0/k_0$ together with a maximal torus and a Borel subgroup $T_0 \subset B_0 \subset G_0$. By $T$, $B$ and $G$, resp., we denote the base change of these groups to $k$. Let $\s$ denote the Frobenius automorphism on $k$ with fixed field $k_0$, and likewise the isomorphism induced on $G(k)$, etc.
In the following, we often implicitly identify a variety over an algebraically closed field $k$ with its $k$-valued points.

We denote by $W$ the Weyl group of $T$, and by $w_0$ its longest element.
Let $\mathbb S\subset W$ denote the set of simple reflections defined by $B$. This defines the Bruhat order and the length function $\ell$ on $W$.

For any $w\in W$, we have the Deligne-Lusztig variety attached to $w$:
\[
X(w) = \{ gB\in G/B;\ g^{-1}\s(g)\in BwB  \}.
\]
This is a smooth locally closed subvariety of the flag variety $G/B$ of dimension $\ell(w)$.

Now $\sigma$ acts on $\mathbb S$ since $T$ and $B$ are defined over $k_0$, and we call an element $w\in W$ a \emph{twisted Coxeter} element, if $w$ can be written as the product of elements of $\mathbb S$ which lie in different $\s$-orbits, and such that every $\s$-orbit occurs among the factors.

\begin{proposition} (Lusztig, \cite{Lusztig} Cor.~2.5)\label{prop-coxeter-dlv}
Let $X\subset G/B$ the Deligne-Lusztig variety attached to a twisted Coxeter element. Then $X\subseteq Bw_0B/B$.
\end{proposition}


\begin{example}
Let us explain the case of $G=GL_n$, $k_0=\BF_p$. We identify the Weyl group for the diagonal torus with the symmetric group in the usual way, and use the transpositions $s_i = (i, i+1)$, $i=1, \dots, n-1$, as Coxeter generators corresponding to the Borel subgroup of upper triangular matrices. As explained in~\cite{Deligne-Lusztig} 2.2, the Deligne--Lusztig variety $X$ for the (twisted) Coxeter element $w = s_1 s_2 \cdots s_{n-1}$ consists of those flags $(\sF_i)_i$ in $k^n$ such that
\[
\sF_i = \sum_{j=1}^i \sigma^j(\sF_1) \qquad\text{for all } i.
\]
Looking at the Moore determinant, we see that a line of the form $\langle \sum a_i e_i \rangle$ gives rise to a point in $X$ (i.e., all sums on the right hand side of the displayed formula are in fact direct sums) if and only if the $a_i$ are linearly independent over $\BF_p$, or in other words, if and only if this line is not contained in a rational hyperplane.

From this, it is easy to see that a point $(\sF_i)_i\in X$ satisfies
\[
\dim \sF_i \cap \langle e_1, \dots, e_j \rangle_k = \max(0, i+j-n).
\]
This is precisely the Schubert condition which ensures that $X$ is contained in the open cell $Bw_0B/B$.
\end{example}

See the paper \cite{Rapoport-Terstiege-Wilson} by Rapoport, Terstiege and Wilson, Section 5, and Wu's paper \cite{Wu}, Examples 4.2, 4.3, 4.6, 4.7, for further examples for other classical groups.

\subsection{Notation}\label{sec:notation}

Now we return to the setting of the introduction: Let $F$ be a non-archimedean local field, $\overline{F}$ a separable closure of $F$, and $\brF$ the completion of its maximal unramified extension $F^{\rm un}$. We denote by $\varepsilon\in F$ a uniformizer, by $\kappa$ the residue class field of $F$, and by $\overline{\kappa}$ the residue class field of $\brF$, an algebraic closure of $\kappa$. Let $\s$ denote the Frobenius automorphism of $\brF$ over $F$.

Let $G/F$ be a connected semisimple group of adjoint type. (It is easy to extend the results below to reductive groups, but we restrict to the key case of semisimple adjoint groups here.) We have the \emph{Kottwitz homomorphism} $\kappa_G\colon G(\brF)\to \pi_1(G)_{\mathop{\rm Gal}(\overline{F}/F^{\rm un})}$, where $\pi_1(G)$ denotes the algebraic fundamental group of $G$. See~\cite{kottwitz-isoII} \S 7, \cite{Rapoport-Richartz}. Its kernel $G(\brF)_1:=\Ker(\kappa_G)$ is the subgroup of $G(\brF)$ generated by all parahoric subgroups. The analogous result also holds over $F$, and we again denote the concerning subgroup by $G(F)_1$. See~\cite{Richarz} Lemma 4.3.

Fix a maximal $\brF$-split torus of $G$ over $F$ and denote by $T$ its centralizer, a maximal torus in $G$ (because $G$ is quasi-split over $\brF$). Let $N$ be the normalizer of $T$ in $G$. These choices give rise to the (relative) finite Weyl group $W_0 = N(\brF)/T(\brF)$ and the Iwahori-Weyl group $\tW = N(\brF)/T(\brF)_1$, where $T(\brF)_1 \subset T(\brF)$ is the unique parahoric subgroup, or equivalently the kernel of the Kottwitz homomorphism for the group $T$.


We denote by $\breve{\mathcal B} = \mathcal B(G, \brF)$ the building of $G$ over $\brF$. Associated with $S\subset G$, there is the ``standard apartment'' $\brsA$, an affine space under the vector space $V = X_*(T)_{\Gal(\overline{F}/F^{\rm un})} \otimes_{\mathbb Z}\mathbb R$ on which $\tW$ acts by affine transformations. The Frobenius $\sigma$ also acts on $\brsA$, and we fix a $\sigma$-invariant alcove $\fka$ which we call the base alcove. We fix a special vertex $0$ in the closure of $\fka$ which is fixed under the Frobenius automorphism of the unique quasi-split inner form of $G$, and using this vertex as the base point, we identify $\brsA$ with $V$. This choice of special vertex also gives rise to a decomposition of $\tW$ as a semi-direct product $\tW = X_*(T)_{\Gal(\overline{F}/F^{\rm un})} \rtimes W_0$ (but note that the inclusion $W_0\subset \tW$ is not $\sigma$-equivariant in general).

Inside $\tW$ we have the affine Weyl group $W_a$; we can define it as the extended affine Weyl group of the simply connected cover of the derived group of $G$. In terms of the Kottwitz homomorphism $\kappa_G$, we can express $W_a$ as $W_a=(N(\brF)\cap G(\brF)_1)/T(\brF)_1$, see~\cite{Richarz}.

From the choice of a base alcove, we obtain a system of simple affine reflections $\tSS\subset W_a$. Then $\tSS$ generates $W_a$ and $(W_a, \tSS)$ is a Coxeter system. We extend the length function and the Bruhat order to $\tW$ in the usual way; an element has length $0$ if and only if it fixes the base alcove $\fka$. For a subset $P\subseteq \tSS$, we denote by $W_P\subseteq W_a$ the subgroup generated by the elements of $P$, and by ${}^P \tW$ the set of minimal length representatives in $\tW$ of the cosets in $W_P\backslash \tW$. 

The set $P$ also gives rise to a standard parahoric subgroup $\brP$: It is the subgroup of $G(\brF)$ generated by $\brI$ and (lifts to $N(\brF)$ of) the elements of $P$.

We denote by $\Phi$ the set of finite roots in the unique reduced affine root system in $V$ attached to the relative affine root system of $G$ over $\brF$, and for $\alpha\in\Phi$ and $k\in\mathbb Z$, we denote by $H_{\alpha, k} = \{ v\in V;\ \langle \alpha, v\rangle = k\}$ the corresponding affine root hyperplane.

We denote by $\Lambda$ the lattice of translations of $\brsA$ which arise from the action of $\tW$. Since $G$ is assumed to be of adjoint type, $X_*(T)_{\Gal(\overline{F}/F^{\rm un})}$ is torsion-free, so that $\Lambda = X_*(T)_{\Gal(\overline{F}/F^{\rm un})}$. For $\lambda\in \Lambda$, we denote by $\varepsilon^\lambda$ the corresponding element in $\tW$.


\subsection{The Bruhat-Tits building of $G$}

We now come back to the building of $G$ and discuss it in a little more detail.
The extended affine Weyl group $\tW$ acts on the set of alcoves of $\breve{\mathcal A}$. Having chosen a base alcove $\fka$, mapping $w\in W_a$ to $w\fka$, we can identify the affine Weyl group $W_a$ with the set of alcoves in $\breve{\mathscr A}$.

The Frobenius $\sigma$ acts on the building $\breve{\mathcal B} = \mathcal B(G, \brF)$ of $G$ over $\brF$. If we consider the building as a metric space (which carries a simplicial structure), then we can identify the set of fixed points of $\sigma$ with the rational building $\mathcal B=\mathcal B\mathcal(G, F)$ of $G$ over $F$. Note that in general the embedding $\mathcal B \subset \breve{\mathcal B}$ is not induced by a morphism of simplicial complexes, i.e., a vertex of $\mathcal B$ does not necessarily map to a vertex of $\breve{\mathcal B}$ (see the example below).

The same situation is, of course, obtained if we pass to another form of $G$ over $F$. In particular, for $\tau\in N(\brF)$ with $\tau\brI\tau^{-1}=\brI$, i.e., $\tau$ gives a length $0$ element in $\tW$, we can consider the twisted Frobenius $\s_\JJ:= \mathop{\rm Int}(\tau)\circ \s: x\mapsto \tau\s(x)\tau^{-1}$ on $G(\brF)$ and the corresponding inner form $\JJ$ of $G$ over $F$ (where the notation is chosen in view of the applications to affine Deligne-Lusztig varieties, see below). In this situation, the affine and extended affine Weyl groups over $F$ can be identified with the subgroups of $W_a$ and $\tW$, respectively, consisting of the elements fixed by the Frobenius automorphism (induced by) $\s_\JJ$. We denote the rational affine Weyl group corresponding to $\JJ$ by $W_a^\JJ$. For the buildings, we get $\mathcal B(\JJ, F)\subset \mathcal B(\JJ,\brF)=\mathcal B(G, \brF)$. See~\cite{Richarz} for a detailed discussion of the relation between the affine Weyl groups over $F$ and over $\brF$. See also \cite{RZ:indag}.

\begin{example}\label{example-rational-building}
Let $G/F$ be a quasi-split unitary group of rank $3$ which splits over an unramified quadratic extension of $F$. The affine Dynkin diagram of $G$ then is a circle with $4$ vertices $0$, $1$, $2$, $3$, on which the Frobenius acts by exchanging $1$ and $3$, and fixing $0$ and $2$. In $\tW$, there is a length $0$ element $\tau$ which acts on $\tSS$ (by conjugation) by mapping $1\mapsto 2\mapsto 3\mapsto 0\mapsto 1$. The twisted Frobenius $\mathop{\rm Int}(\tau)\circ \s$ exchanges $0$ and $1$, and $2$ and $3$. 

The fix points of the twisted Frobenius inside the base alcove are the points on the line from the midpoint of the edge connecting $0$ and $1$ to the midpoint of the edge connecting $2$ and $3$. See the figure. In particular, the origin of the apartment over $\brF$ does not lie in the rational apartment. (As the origin of the rational apartment we choose the unique vertex of the rational base alcove such that closure of corresponding $\brF$-face contains the $\brF$-origin.)

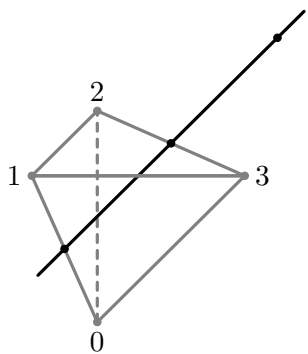
\begin{figure}
\begin{tikzpicture}[line join = round, line cap = round, scale=0.7]
\pgfmathsetmacro{\factor}{1/2};
\coordinate [label=below:0] (D) at (0,-2,2*\factor);
\coordinate [label=left:1] (B) at (-2,0,-2*\factor);
\coordinate [label=above:2] (C) at (0,2,2*\factor);
\coordinate [label=right:3] (A) at (2,0,-2*\factor);

\coordinate (E) at (-1, -1, 0);
\coordinate (F) at (1, 1, 0);
\draw[dashed, gray, very thick] (D)--(C);

\draw[very thick] (E)--(-1.5, -1.5, 0);
\draw[very thick] (3.5, 3.5, 0)--(F);
\draw[very thick] (E)--(F);

\draw[-, gray, very thick] (A)--(D)--(B)--cycle;
\draw[-, gray, very thick] (A) --(C)--(B)--cycle;

\draw[fill=black] (E) circle (.07);
\draw[fill=black] (F) circle (.07);
\draw[fill=black] (3, 3, 0) circle (.07);

\draw[fill=gray, color=gray] (A) circle (.07);
\draw[fill=gray, color=gray] (B) circle (.07);
\draw[fill=gray, color=gray] (C) circle (.07);
\draw[fill=gray, color=gray] (D) circle (.07);

\end{tikzpicture}
\caption{The base alcove in the building $\mathcal B(G, \brF)$ (gray lines) and part of the standard apartment in the rational building $\mathcal B(\JJ, F)$ (black line).}
\end{figure}
\end{example}

\subsection{The gate property}\label{sec:gate}

In this section we recall a number of building theoretic concepts and results. Our reference is the book \cite{AB} by Abramenko and Brown.

For any two alcoves $\fkx$, $\fky$ in $\breve{\mathcal B}$, we have the relative position or \emph{Weyl distance} $\delta(\fkx, \fky)\in W_a$, which we can obtain by choosing an apartment containing $\fkx$ and $\fky$, identifying the set of alcoves in the apartment with the affine Weyl group, and setting $\delta(\fkx, \fky) = x^{-1}y$, where $x, y\in W_a$ are the elements corresponding to $\fkx$ and $\fky$, resp., under this identification. This element is independent of all choices. Let $(t_1, \dots, t_r)$ with $t_i\in \tSS$ be the type of a gallery from $\fkx$ to $\fky$ which is contained in one apartment (which is always the case if the gallery is minimal). Then $\delta(\fkx, \fky) = t_1 \cdots t_r$.

For $K\subset \tSS$, we also define the \emph{$K$-Weyl distance} $\delta_K(\fkx, \fky):= W_K\delta(\fkx, \fky) W_K \in W_K\backslash \tW/W_K$. With this notation, $\delta = \delta_\emptyset$.

If $\mathcal M$ and $\mathcal N$ are sets of alcoves, we define
\[
\delta(\mathcal M, \mathcal N) := \{ \delta(\fkx, \fky);\ \fkx\in \mathcal M,\ \fky\in \mathcal N\}.
\]

We also define the distance $\dist(\fkx, \fky):=\ell(\delta(\fkx, \fky))$ between alcoves $\fkx$ and $\fky$ as the length of their Weyl distance.

Given an alcove $\fkx$ and a set $S$ of alcoves, we write $\dist(\fkx, S):= \min_{\fky\in S} \dist(\fkx, \fky)$.

\begin{definition} (\cite{AB}, Def.~5.26)
Let $P\subseteq \tSS$. For an alcove $\fkx$ we call the set
\[
\{ \fky;\ \delta(\fkx, \fky)\in W_P \} \]
the $P$-residue of $\fkx$. A set of alcoves is called a $P$-residue, if it is of the above form for some alcove $\fkx$. The set $P$ is called the type of the residue.
\end{definition}

It is easy to see that every residue has a unique type. In fact, we have $\delta(\CR, \CR) = P$, and more generally one easily checks:

\begin{lemma}{\rm (\cite{AB}, Lemma 5.29)}
Let $\CR$ be a $P$-residue, and $\CR'$ a $P'$-residue for subsets $P, P'\subseteq \tSS$. Then there exists $w\in W_a$ such that
\[
\delta(\CR, \CR') = W_P w W_{P'}.
\]
\end{lemma}

\begin{example}
For $P\subset\tSS$ with associated parahoric subgroup $\brP$, we can consider the quotient $\brP/\brI$ as a set of alcoves, and this set is a $P$-residue. If $\mathcal R$ is a residue of type $P$, and $g\in G(\brF)_1$, then $g\mathcal R$ is also a residue of type $P$ (since the action of $G(\brF)_1$ on the building is type-preserving). All residues of type $P$ have the form $g \brP/\brI$ for some $g\in G(\brF)_1$.
\end{example}

A key property of residues is the following proposition, called the gate property. The main ingredient in its proof, see loc.~cit., is the following fact: For $w\in\tW$ and $P, P'\subseteq \tSS$, there exists a unique element of minimal length in the double coset $W_PwW_{P'}$ which we denote by $\min(W_PwW_{P'})$; see~\cite{AB} Prop. 2.23.

\begin{proposition}{\rm (\cite{AB}, Prop.~5.34)}
Let $\CR$ be a residue, and $\fkb$ an alcove. There exists a unique alcove $\fkg\in\CR$ such that
\[
\dist(\fkb, \fkg) = \min \{ \dist(\fkb, \fkx);\ \fkx\in \CR \}.
\]
The alcove $\fkg$ is called the gate from $\fkb$ to $\CR$ and is denoted by $\proj_\CR(\fkb)$. It has the following properties:
\begin{enumerate}
\item
$\delta(\fkb, \fkg)$ is the unique element of minimal length in $\delta(\fkb, \CR)$.    
\item
For all alcoves $\fkx\in\CR$, we have
\[
\delta(\fkb, \fkx) = \delta(\fkb, \fkg) \delta(\fkg, \fkx).
\]
\item
For all alcoves $\fkx\in\CR$, we have
\[
\dist(\fkb, \fkx) = \dist(\fkb, \fkg) +  \dist(\fkg, \fkx).
\]
\end{enumerate}
\end{proposition}

We obtain a ``projection'' $\proj_\CR$ from the set of all alcoves in $\mathcal B(G, \brF)$ to $\CR$ by mapping each alcove $\fkb$ to the gate from $\fkb$ to $\CR$. With the notation of the proposition, we can interpret part (2) as saying that for every $\fkx\in \CR$, there exists a minimal gallery from $\fkb$ to $\fkx$ passing through $\fkg$. Note however that usually there will also exist other minimal galleries from $\fkb$ to $\fkx$ not passing through $\fkg$.

\begin{remark}\label{rmk-gate-rational}
The uniqueness of the gate from $\fkb$ to $\CR$ implies that every automorphism which preserves $\fkb$ and $\CR$ also fixes the gate from $\fkb$ to $\CR$. In particular, if $\fkb$ and $\CR$ are defined over $F$ (i.e., stable under all Galois automorphisms), then so is the gate.
\end{remark}

We will also need the following generalization of the gate property:

\begin{proposition}{\rm (\cite{AB}, Prop.~5.37)}\label{AB-5.37}
Let $\CR$ be a $P$-residue, and $\CR'$ a $P'$-residue for subsets $P, P'\subseteq \tSS$. Let $w_1 = \min(\delta(\CR, \CR'))$, $\CR_1 = \proj_{\CR}(\CR')$, $\CR'_1 = \proj_{\CR'}(\CR)$.
\begin{enumerate}
\item
The set $\CR_1$ is a residue of type $P \cap w_1 P' w_1^{-1}$, and likewise $\CR'_1$ is a residue of type $w_1^{-1}Pw_1 \cap P'$.
\item
The maps $\proj_{\CR'|\CR_1}\colon \CR_1\to \CR'_1$ and $\proj_{\CR|\CR'_1}\colon \CR_1'\to \CR_1$ are inverse to each other and hence induce a bijection $\CR_1 \cong \CR'_1$.
\item
Let $\fkx\in\CR_1$. The alcove $\fkx'\in\CR'_1$ corresponding to $\fkx$ under the bijection in (2) is the unique alcove in $\CR'_1$ such that $\delta(\fkx, \fkx') = w_1$.
\end{enumerate}
\end{proposition}

\subsection{Affine Deligne-Lusztig varieties}\label{sec:adlv}

We fix a triple $(G, K, \mu)$, where $G$ is a connected semisimple group of adjoint type over $F$, as above, $K\subset \tSS$ and $\mu\in\Lambda$ is a translation element of the corresponding affine root system.

Attached to $\mu$, we have the $\mu$-admissible set
\[
\Adm(\mu) := \{ w\in\tW;\ \text{there exists } v\in W_0: w\in \varepsilon^{v(\mu)} \}.
\]

From $K\subset \tSS$ we obtain the parahoric subgroup $\brK\subset G(\brF)$, and in case $K$ is stabilized by $\sigma$, a rational parahoric subgroup $\mathcal K\subset G(F)$. We denote the corresponding partial affine flag variety by $\Flag_K$ (this is the ind-scheme representing the sheaf quotient of the loop group of $G$ by the positive loop group scheme corresponding to $\brK$). For $K=\emptyset$, $\brK$ is the standard Iwahori subgroup $\brI$ which we fixed by fixing a base alcove in the standard apartment. We write $\Flag$ for $\Flag_{\emptyset}$. In the case of mixed characteristic, the notion of ind-scheme here has to be understood in the setting of perfect schemes, see~\cite{Zhu}, \cite{Bhatt-Scholze}.

The \emph{affine Deligne-Lusztig variety} attached to $w\in \tW$ and $b\in G(\brF)$ is
\[
X_w(b) = \{ g\in G(\brF)/\brI;\  g^{-1}b\s(g)\in \brI w\brI \}.
\]
More precisely, the right hand side denotes the set of $\overline{\kappa}$-valued points of a unique reduced locally closed subscheme of the affine flag variety $\Flag$, which we also denote by $X_w(b)$. As before, in the mixed characteristic case, $X_w(b)$ is a perfect scheme. Then $X_w(b)$ is (perfectly) locally of finite type over $\overline{\kappa}$ and is finite-dimensional. 

We also consider the following variants: For $K\subseteq \tSS$ and $w\in W_K\backslash \tW/W_K$, we have
\[
X_w(b) = \{ g\in G(\brF)/\brK;\  g^{-1}b\s(g)\in \brK w\brK \},
\]
where again $b\in G(\brF)$. Often it is better to consider the space
\[
X(\mu, b)_K = \{ g\in G(\brF)/\brK;\ g^{-1}b\s(g)\in \bigcup_{w\in \Adm(\mu)} \brK w\brK \},
\]
a finite disjoint union of usual affine Deligne-Lusztig varieties which is more directly related to the corresponding Rapoport-Zink space if the group-theoretic data arises from a suitable (local) Shimura datum, see~\cite{rapoport:guide}.

In the case considered in \cite{CV}, $K = \tSS\setminus \{0\}$ corresponds to a hyperspecial parahoric group and $\mu$ is minuscule, and then $X(\mu, b)_K$ is an affine Deligne-Lusztig variety in the usual sense, because the union $\bigcup_{w\in \Adm(\mu)} \brK w\brK$ is equal to the $\brK$-double coset $\brK \varepsilon^\mu \brK$. In the general case, we will formulate the results in Section 3 for the union $X(\mu, b)_K$, as in \cite{GH}, but it will be clear from the statements and proofs that we could also consider the individual affine Deligne-Lusztig varieties in the union $X(\mu, b)_K$ separately.

\begin{definition}
Let $b\in G(\brF)$.
We denote by $\JJ_b$ the $\sigma$-centralizer of $b$:
\[
\JJ_b(R) = \{ g\in G(R\otimes_F \brF);\ g^{-1} b \sigma(g) = b \}
\]
(where $R$ is any $F$-algebra). This defines an algebraic group $\JJ_b$ over $F$. We usually fix $b$ and write $\JJ = \JJ_b(F)$.
\end{definition}

The element $b$, or its $\sigma$-conjugacy class, is called \emph{basic}, if the algebraic group $\JJ$ is an inner form of $G$. See~\cite{kottwitz-isoI}, \cite{Rapoport-Richartz}, \cite{GHN2} Section 1.2 for a more detailed discussion.

\section{The $\JJ$-stratification}

\subsection{Definition of the $\JJ$-stratification}

To state the definition of the \CVstrat{}, we introduce the relative position map:

\begin{definition}\label{def-relative-position}
Let $K\subset \tSS$, and let $\brK$ be the corresponding parahoric subgroup. We denote by $\inv_{K}$ the \emph{relative position map}
\[
\inv_K \colon G(\brF) \times G(\brF) \to \brK \backslash G(\brF) \backslash \brK \cong W_K \backslash \tW /W_K,\quad (g,h) \mapsto \brK g^{-1}h \brK,
\]
where the identification $\brK \backslash G(\brF) \backslash \brK \cong W_K \backslash \tW /W_K$ comes from the ``parahoric Iwahori-Bruhat decomposition'' for $\brK$. For $K=\emptyset$, $\brK=\brI$, this is the usual Iwahori-Bruhat decomposition, for $\brK$ hyperspecial, it is the Cartan decomposition. See~\cite{Haines-Rapoport} Prop.~8.
\end{definition}

The map $\inv_K$ factors through $\Flag\times \Flag$ (and even through $\Flag_K\times \Flag_K$).

\begin{remark}
Note that the relative position map is closely related to the Weyl distance $\delta$ defined above. In fact, if $G(\brF)_1$ denotes the kernel of the Kottwitz homomorphism, then the action of $G(\brF)$ on the building $\breve{\mathcal B}$ gives us a commutative diagram
\[
\xymatrix{
G(\brF)_1 \times G(\brF) \ar[d]\ar[r]^>>>>>>{\inv} & \tW \ar[d] \\
{\rm Alc}(\breve{\mathcal B})\times {\rm Alc}(\breve{\mathcal B}) \ar[r]^>>>>>\delta & W_a,
}
\]
where ${\rm Alc}(\breve{\mathcal B})$ denotes the set of alcoves in $\breve{\mathcal B}$, the vertical map on the left is $(g, h)\mapsto (g\fka, h\fka)$, and the vertical map on the right is the projection $\tW = W_a\rtimes \{ w\in\tW;\ \ell(w)=0\} \to W_a$. We need to restrict the values in the first factor of the top left corner to $G(\brF)_1$, because the action of $G(\brF)$ is not type-preserving in general.
Projecting to the double quotients by $W_K$, we get a similar diagram for $\inv_K$ and $\delta_K$, $K\subset\tSS$.
\end{remark}

\begin{definition}{\rm (\cite{CV} Section 2)}
Fix $b\in G(\brF)$ and denote by $\JJ_b$ its $\s$-centralizer, $\JJ:=\JJ_b(F)$.
For $K\subset \tSS$ and a family $\mathbf w = (w_j)_{j\in \JJ}$, we let
\[
S_\JJ(\mathbf w) = \{ x\in G(\brF)/\brK;\ \forall j\in \JJ: \inv_K(j, x) = w_j \}.
\]
The strata of the $\JJ$-stratification (of $G(\brF)/\brK = \Flag_K(\overline{\kappa})$) are those sets $S_\JJ(\mathbf w)$ which are non-empty. By intersecting $S_\JJ(\mathbf w)$ with $X(\mu,b)_K$, we obtain the $\JJ$-stratification on $X(\mu, b)_K$.
\end{definition}

As explained in \cite{CV} Section 2, it follows from Theorem~\ref{thm-finiteness} below that each stratum is equal to the set of $\overline{\kappa}$-valued points of a (unique) reduced locally closed subscheme of $\Flag_K$, and we usually identify the set of $\overline{\kappa}$-valued points and this scheme.

\subsection{A finiteness property of the $\JJ$-stratification}\label{sec:finiteness-property}

We start with some general preparations which are independent of the group $\JJ$.

\begin{definition} 
Let $\mathfrak x$ be an alcove in the apartment $\breve{\mathscr A}$. We denote by  $\mathscr D_R(\mathfrak x)$ the set of walls of $\fkx$ separating $\fkx$ from the base alcove $\fka$.
\end{definition}

So the elements if $\mathscr D_R(\mathfrak x)$ are the affine root hyperplanes which intersect the closure of $\mathfrak x$ in a face of codimension $1$. Their types in $\tSS$ form the right descent set $D_R(\mathfrak x)$ of $\mathfrak x$ (considered as an element of the affine Weyl group).

Recall that in our terminology the maximal simplices in the affine building are called alcoves, while the notion of chamber is used for the finite Weyl chambers, i.e., the chambers of the spherical building given by our choice of origin.

If $v$ is any vertex in the standard apartment $\breve{\mathscr A}$, the collection of affine root hyperplanes passing through $v$ is a (finite) root system in the vector space with origin $v$ obtained from $\breve{\mathscr A}$. We call this the root system obtained by change of origin to $v$. Note that if $v$ is not special, then this root system will have a different Dynkin type. Cf.~\cite{GHN2} 5.7.

\begin{lemma}
Let $\fkx$ be an alcove in the standard apartment $\breve{\mathscr A}$, and let $H$ be an affine root hyperplane which is a wall of $\fkx$. Let $v$ be a vertex of $\fkx$ not lying in $H$. Let $C$ be the chamber opposite to the chamber containing $\fkx$ with respect to the root system given by change of origin to $v$.

Let $\mathscr H$ be a finite set of affine root hyperplanes in $\breve{\mathscr A}$.

Then there exists a translation $\lambda\in \Lambda$ with the following property:
For all alcoves $\fky = \varepsilon^\mu \fkx$ (for some $\mu\in \Lambda$) such that $H$ is a wall of $\fky$ and such that $\fky$ is not contained in $\varepsilon^\lambda C$,
\[
\mathscr D_R(\fky) \not\subseteq \mathscr H.
\]
\end{lemma}

Note that typically $\mathscr H$ will contain $H$ and $H\in \mathscr D_R(\fkx)$ (otherwise the statement is trivially satisfied). If the affine Dynkin diagram of $G$ is connected, then $v$ in the lemma is unique; but otherwise, alcoves will be products of simplices (polysimplices), and $v$ is not unique.

\begin{proof}
We first prove the lemma in the case that $\mathscr H = \{ H\}$, and will discuss the general case afterwards. In this case, we choose $\lambda$ such that $\varepsilon^\lambda C$ contains the base alcove $\fka$.

Let us check that the conclusion of the lemma is then satisfied.
Denote by $v'$ the vertex of $\fky$ obtained by translating $v$ by $\mu$, and denote by $C'$ the ``chamber'' with vertex $v'$ which contains $\fky$. (So up to translation $C$ and $C'$ are opposite to each other.)
Because $v'\not\in \varepsilon^\lambda C$, one of the walls of $\varepsilon^\lambda C$ separates $v'$ from the base alcove, say $H_{\alpha, k}$. Replacing $\alpha$ by $-\alpha$ and $k$ by $-k$, if necessary, we may assume that $k\ge 0$. Let $k'$ be such that $H_{\alpha, k'}$ passes through $v'$ (and hence is a wall of $\fky$). Then $H_{\alpha, k'}$ has ``larger distance'' to the base alcove than $H_{\alpha, k}$, i.e., $k' \ge k$, and $H_{\alpha, k'}$ separates $\fky$ from the base alcove (because up to translation the ``chambers'' $\varepsilon^\lambda C$ and the one with apex $v'$ and containing $\fky$ are opposite to each other). Thus $H_{\alpha, k'}\in \mathscr D_R(\fky)\setminus \{ H\}$.

We now discuss how to modify the choice of $\lambda$ in order to deal with the general case. By changing $\lambda$, we may replace the number $k$ in the above argument by any given larger number. The walls of the alcoves $\fky$ not in $\varepsilon^\lambda C$ which we produce will have the form $H_{\alpha, k'}$ with $k'\ge k$. For $k$ sufficiently large, no $H_{\alpha, k'}$ can be in $\mathscr H$. Choosing $\lambda$ which works simultaneously for all $\alpha$, it is then clear that we may ensure that the wall of $\fky$ we find is not in the set $\mathscr H$.
\end{proof}

\begin{proposition}\label{prop-finiteness}
Let $\mathscr H$ be a finite set of affine root hyperplanes in the apartment $\breve{\mathscr A}$. Then there are only finitely many alcoves $\fkx$ in this apartment with $\mathscr D_R(\fkx) \subseteq  \mathscr H$.
\end{proposition}

\begin{proof}
Let $H\in \mathscr H$. Since $\mathscr H$ is finite and the set $\mathscr D_R(\fkx)$ is empty only if $\fkx$ is the base alcove, it is enough to show that there are only finitely many alcoves $\fkx$ with
\[
H\in \mathscr D_R(\fkx) \subseteq \mathscr H.
\]
Let $\mathcal M$ denote the set of all of these alcoves. Call two alcoves in $\mathcal M$ equivalent, if one is mapped to the other by some translation element. In view of the above lemma, it is enough to show that there are only finitely many equivalence classes for this equivalence relation. (Because, with the notation of the lemma, only finitely many alcoves which have $H$ as a wall lie in all the finitely many chambers $\varepsilon^\lambda C$ produced by the lemma for $\fkx$ and the different choices for $v$.)

This finiteness assertion follows from the fact that $H\cap \Lambda$ is a lattice of full rank in the space $H$: Hence a fundamental mesh for this lattice is bounded, and meets only finitely many alcoves.

In fact, every finite root is part of some system of simple roots for the root system. Say $H=H_{\alpha, k}$ for $\alpha\in \Phi$. Choosing a system of simple roots $\alpha = \beta_1, \beta_2, \dots, \beta_r$, and denoting the corresponding fundamental coweights by $\omega^\vee_1, \dots, \omega^\vee_r$, we find linearly independent points $k\omega^\vee_1 + \omega^\vee_i$, $i=2, \dots, r$, in $H\cap \Lambda$.
\end{proof}


\begin{corollary}\label{cor-separate-1}
Given a bounded subset $S \subset \breve{\mathcal B}$, there exists $c > 0$ such that for all alcoves $\fkj$ in $\breve{\mathcal B}$ with $\dist(\fkj, \fka) > c$, there exists an adjacent alcove $\fkj'$ of the alcove $\fkj$ such that every alcove in $S$ can be reached from $\fkj$ by a minimal gallery passing through $\fkj'$.
\end{corollary}

\begin{proof}
We may replace $S$ by a larger subset, so we may assume that $S$ is the set of all alcoves of distance $\le d_S$ to the base alcove $\fka$, for some $d_S$. Let $\fkj$ be an alcove with $\dist(\fkj, \fka)>d_S$, and let $\mathscr A'$ be an apartment containing $\fkj$ and the base alcove $\fka$. Every minimal gallery from $\fkj$ to an alcove in $S$ passes through an alcove $\fky$ in $\mathscr A'$ with $\dist(\fkj, \fky)=d_S$. In fact, consider such a minimal gallery. Looking at the retraction from the building onto $\mathscr A'$ fixing $\fka$ (cf.~\cite{AB} Def.~4.38), we see that as soon as this gallery leaves $\mathscr A'$, the distance to $\fka$ increases with each step. Therefore it is enough to find a wall of $\fkj$ such that the half-space of $\mathscr A'$ containing $\fkj$ contains none of the alcoves in $S$.

By the defining axioms of buildings (\cite{AB}~Def.~4.1), there is an isomorphism $\mathscr A' \cong \brsA$ fixing $\fka$ pointwise, we may just as well assume that $\mathscr A' = \brsA$. Let $c_1 \ge 1$ be so large that for all affine root hyperplanes $H_{\alpha, k}$ with $|k| > c_1$, all alcoves in $S$ contained in $\brsA$ lie on the same side of $H_{\alpha, k}$.

Now apply Prop.~\ref{prop-finiteness} in $\breve{\mathscr A}$ for the finite set $\mathscr H$ of affine root hyperplanes $H_{\alpha, k}$ with $|k| \le c_1$. Choose $c>d_S$ such that $\dist(\fkj, \fka) \le c$ for all $\fkj$ with $\mathscr D_R(\fkj)\subseteq \mathscr H$.
\end{proof}

\begin{theorem}\label{thm-finiteness}
Let $b\in G(\brF)$, let $\JJ_b$ be its $\s$-centralizer, and let $\JJ=\JJ_b(F)$.
Fix a parahoric level $K\subset \tSS$ and a quasi-compact subscheme $S\subset \Flag_K$ in the partial flag variety for $K$.

There exists a finite subset $J'\subset \JJ$ such that for all $g, g'\in S$ with
\[
\inv_K(j, g) = \inv_K(j, g')\ \text{for all } j\in J',
\]
we have
\[
\inv_K(j, g) = \inv_K(j, g')\ \text{for all } j\in \JJ.
\]
\end{theorem}

Equivalently, we can express the theorem as saying that there exist only finitely many families $(w_j)_{j\in\JJ}$ of the form $(\inv(j, g))_j$ with $g\in S$, i.e., only finitely many $\JJ$-strata intersect $S$.

For $G$ unramified, and $K=\tSS\setminus \{0\}$ hyperspecial, this is \cite{CV} Prop.~2.6. We give a new proof of this result which works for general $G$ and which proceeds solely in terms of the building of $G$.

\begin{proof}
We begin by some reduction steps which simplify the situation. Since the inverse image of $S$ in the full affine flag variety is again a quasi-compact subscheme, and the invariant $\inv = \inv_{\emptyset}$ for the Iwahori subgroup is finer than $\inv_K$, it is enough to prove the theorem in the case $K=\emptyset$.

Assume that we can show the existence of a finite set $J' \subseteq G(\brF)$ with the property of the theorem, but without requiring $J'\subset \JJ$. It still follows that there are only finitely many families $(\inv(j, g))_{j\in \JJ}$ with $g\in S$. These finitely many families can then be separated by finitely many suitably chosen elements of $\JJ$, so that we obtain the theorem as stated.

In particular, this implies that we may replace $\JJ = \JJ_b(F)$ by a larger group. For $s$ sufficiently large, the base change of $\JJ_b$ to the unramified extension $F_s$ of $F$ of degree $s$ is a subgroup of $G_{F_s}$ (more precisely, the centralizer of a $1$-parameter subgroup depending on $b$). It is therefore enough to prove the statement for $G(F_s)$ in place of $\JJ$. Increasing $s$ further, we may assume that $\s^s$ acts trivially on $\tSS$ and fixes the standard apartment $\breve{\mathscr A}$, i.e., $G_{F_s}$ is residually split. To ease the notation, we keep the notation $\JJ$ and assume that it is residually split over $F$.

Furthermore, it is easy to check that $\inv(j, g)=\inv(j, g')$ for all $j\in\JJ$ if this is true for all $j\in \JJ_1$. It follows that we can just as well replace
$S$ by the set $\{ s\fka;\ s\in S(\overline{\kappa})\}$ of alcoves in $\mathcal B(G, \brF)$, and $\inv(j, g)$ by the Weyl distance $\delta(j\fka, g\fka)$. It is then enough to show the existence of a finite set $J'$ of alcoves $\fkj$ in $\mathcal B(G, \brF)$ such that the values $\delta(\fkj, g\fka)$ for $\fkj\in J'$ determine all $\delta(\fkj, g\fka)$ for $\fkj$ in $\mathcal B(\JJ, F)$.

\medskip
\emph{Claim.} There exists a finite set $J'$  of alcoves in $\mathcal B(G, \brF)$ such that

\vspace{1ex}\par
\hfill\parbox{\dimexpr \textwidth-2cm}
{for every alcove $\fkj$ in $\mathcal B(\JJ, F)$ there exists an alcove $\fkj'\in J'$ such that for every $\fkx$ in $S$ there exists a minimal gallery from $\fkj$ to $\fkx$ passing through $\fkj'$.
}
\hfill\llap{(*)}\vspace{3ex}\par

The claim implies the theorem: for $\fkj\in \mathcal B(\JJ, F)$ and $\fkj'\in J'$ chosen with this property, and $\fkx\in S$,
\[
\delta(\fkj,\fkx)  = \delta(\fkj, \fkj') \delta(\fkj', \fkx),
\]
so if $\fkx, \fkx' \in S$ have $\delta(\fkj', \fkx) = \delta(\fkj', \fkx')$, it follows that $\delta(\fkj, \fkx)=\delta(\fkj, \fkx')$. Therefore $J'$ has the desired property.

To prove the claim, let $J'$ be the set of alcoves in $\mathcal B(\JJ, F)$ of distance $\le c$ to the base alcove, where $c$ is as in Cor.~\ref{cor-separate-1}. Then $J'$ is a finite set, because we can view it as a subset of the set of $\kappa$-valued points of some Schubert variety over the finite field $\kappa$. To show that $J'$ satisfies the above property (*), it is enough to prove that the alcove $\fkj'$ produced by the corollary is again a $\JJ$-rational alcove, i.e., lies in $\mathcal B(\JJ, F)$. Then we can inductively apply the corollary until we reach an alcove of distance $\le c$ to $\fka$. Looking at the proof of the corollary, we see that in our situation we can work with a rational apartment $\mathscr A'$ containing $\fkj$ and the base alcove. Since $\JJ$ is residually split by our previous reductions, every $\brF$-alcove in $\mathscr A'$ lies in $\mathcal B(\JJ, F)$, and we are done.
\end{proof}

\section{The Coxeter case}
\label{sec:three}

\subsection{The setting}\label{sec:coxeter-setup}

We begin by recalling some of the results of \cite{GH}. We start with a group $G$ over a local non-archimedean field $F$ as above, see Section~\ref{sec:notation}, a translation element $\mu\in\Lambda$ and a ``level structure'' $K\subset \tSS$. We assume that the affine Dynkin diagram of $G$ is connected.

Let $K = \tSS \setminus \{ v\}$ be a subset fixed by $\s$, i.e., the parahoric subgroup $\brK$ is a rational, maximal parahoric subgroup.

We are interested in the geometry of the spaces $X(\mu, b)_K$, as defined in Section~\ref{sec:adlv}, the most interesting case being the case where $b$ is basic.


For $w\in W_a$, we denote by $\supp(w)\subseteq \tSS$ the set of simple affine reflections occurring in a reduced expression for $w$. Note that this set is the same for each reduced expression; it is called the support of $w$. Let $\tau\in\tW$ be an element of length $\ell(\tau)=0$. Then $\tau$ acts on $\tSS$ by conjugation. We define the $\s$-support $\supps(w\tau)$ of $w\tau$ as the smallest $\t\s$-stable subset of $\tSS$ which contains $\supp(w)$, or equivalently, as the smallest $\t\s$-stable subset $P$ of $\tSS$ with $w \in W_P$.

We call an element $w\in \tW$ a $\s$-Coxeter (or twisted Coxeter) element (with respect to its $\sigma$-support), if exactly one simple reflection from each $\t\s$-orbit on $\supps(w)$ occurs in every (equivalently, some) reduced expression for $w$. Note that we do not require that every $\t\s$-orbit in $\tSS$ occurs in $w$; for example, the identity element is $\s$-Coxeter with respect to its $\s$-support. So there is a slight difference to the terminology of Section~\ref{sec:lusztig}.

Attached to these data, we have the set
\[
\EOcox:= \{ w\in \mathop{\rm Adm}(\mu) \cap {}^K\tW;\ \supps(w) \ne \tSS,\ \text{and } w\ \text{is $\s$-Coxeter for } \supps(w)\}. 
\]
Here EO stands for Ekedahl-Oort; the above set is the index set for the set of ``Ekedahl-Oort strata'' contained in the basic locus.

Now assume that the triple $(G, \mu, K)$ is ``of Coxeter type'' as defined in~\cite{GH} Section 5, i.e., we require that for $b$ basic,
\[
X(\mu, b)_K = \bigsqcup_{w\in \EOcox} \pi(X_w(b)),
\]
where $\pi\colon \Flag\to \Flag_K$ is the projection. This is a strong condition with a number of pleasant consequences. Loc.~cit.~Theorem 5.1.2 contains a complete classification of these triples. Also compare with \cite{GHN2}.

This implies that for $b$ non-basic, $X(\mu, b)_K$, if non-empty, is equal to a union of $\pi(X_w(b))$ for elements $w\in \Adm(\mu)\cap {}^K\tW$ which are $\sigma$-straight (see~\cite{GH} Section 1.4), and $\dim X_w(b) =0$.

For $b$ basic, $w\in \Adm(\mu)\cap {}^K\tW$, we have $X_w(b)\ne \emptyset$ if and only if $w\in \EOcox$; the resulting decomposition is discussed in more detail in the next section. As proved in \cite{GH}, we have a very explicit description of the elements in $\EOcox$ in terms of the affine Dynkin diagram of $G$. For our purposes, we note the following consequences: If for $w, w'\in \EOcox$ we have $\supps(w)=\supps(w')$, then $w=w'$. We identify the set of vertices of the affine Dynkin diagram with $\tSS$, and for $w\in \EOcox$ denote by $\Sigma_w\subset\tSS$ the set of vertices of the affine Dynkin diagram which are adjacent to $\supps(w)$ (but do not lie in $\supps(w)$). If $\supps(w)=\emptyset$, i.e., $\ell(w)=0$, we define $\Sigma_w$ to be the $\tau\sigma$-orbit of $v$. It is shown in loc.~cit., that $\Sigma_w$ determines $w$, and is either a single $\tau\sigma$-orbit, or the union of two $\tau\sigma$-orbits.

Instead of repeating further details from \cite{GH}, we spell out the above objects in two examples:

\begin{example}\label{example1}
We start with the case of a unitary group which splits over an unramified quadratic extension, with $\mu = \omega_1^\vee$ (``signature $(n-1, 1)$''), cf.~\cite{GH} 6.3. This is the case considered by Vollaard and Wedhorn~\cite{Vollaard-Wedhorn}. The affine Dynkin diagram of type $\tilde{A}_{n-1}$ is a circle with vertices $0$, $1$, \dots, $n-1$, on which $\sigma$ acts by fixing $0$, and exchanging $i$ and $n-i$ for $i=1, \dots, n-1$. (So if $n$ is even, $\sigma$ has another fix point at $n/2$.) As level structure, we choose $K = \tSS \setminus \{ 0\}$; the corresponding parahoric subgroup is hyperspecial.

Conjugation by $\tau$ acts by $i\mapsto i+1$ and $n-1\mapsto 0$, so the composition $\tau\sigma$ exchanges $0$ and $1$, $2$ and $n-1$, etc. Cf.~Example~\ref{example-rational-building} above (where $n=4$).

The elements in $\EOcox$ are the elements $\tau$, $s_0\tau$, \dots, $s_0 s_{n-1}\cdots s_{\lceil\frac{n+3}{2}\rceil}\tau$, and the corresponding sets $\Sigma_w$ are the $\tau\sigma$-orbits in $\tSS$, i.e., $\{0, 1\}$, \dots, $\{ \lceil\frac{n}{2}\rceil , \lceil\frac{n+1}{2}\rceil \}$.

In the figure, the situation for $n=9$, $w=s_0s_8\tau$, and $\Sigma_w = \{ 3, 7 \}$ (the two vertices marked with gray circles), $\supps(w) = \{ 8, 0, 1, 2\}$ (the vertices marked with black circles) is shown.

\begin{figure}
\centering
    
\begin{tikzpicture}[scale=2]
\draw[fill=black] 
(0.766044443118978, 0.6427876096865393)
      circle [radius=.05] node [label=$7$] {} --
(0.17364817766693041, 0.984807753012208)
      circle [radius=.05] node [label=$8$] {} --
(-0.4999999999999998, 0.8660254037844387)
      circle [radius=.05] node [label=$0$] {} --
(-0.9396926207859083, 0.3420201433256689)
      circle [radius=.05] node [label=$1$] {} --
(-0.9396926207859084, -0.34202014332566866)
      circle [radius=.05] node [left, label=$2$] {} --
(-0.5000000000000004, -0.8660254037844384)
      circle [radius=.05] node [label=$3$] {} --
(0.17364817766692997, -0.9848077530122081)
      circle [radius=.05] node [label=$4$] {} --
(0.7660444431189778, -0.6427876096865396)
      circle [radius=.05] node [left, label=$5$] {} --
(1.0, 0.0)
      circle [radius=.05] node [right, label=$6$] {} --
(0.766044443118978, 0.6427876096865393)
; 
\draw[fill=gray]
(0.766044443118978, 0.6427876096865393)
      circle [radius=.05]     
;
\draw[fill=gray]
(-0.5000000000000004, -0.8660254037844384)
      circle [radius=.05]     
;

\draw[fill=white]
(1.0, 0.0)
      circle [radius=.05]     
;
\draw[fill=white]
(0.17364817766692997, -0.9848077530122081)
      circle [radius=.05]     
;
\draw[fill=white]
(0.7660444431189778, -0.6427876096865396)
      circle [radius=.05]     
;



\end{tikzpicture}
\caption{Example 3.1 for $n=9$, $w=s_0s_8\tau$, $\Sigma_w= \{3, 7\}$.}
\end{figure}
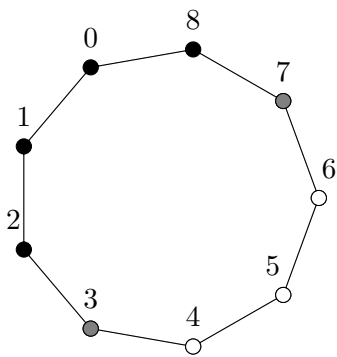
\end{example}

\begin{example}\label{example2}
Second, we give an example where one of the sets $\Sigma_w$ is the union of two $\tau\sigma$-orbits. This case arises from a unitary group which splits over a ramified quadratic extension, see~\cite{GH} 6.3, 7.1.2, 7.4.2 and~\cite{Rapoport-Terstiege-Wilson} for an analysis of the Rapoport-Zink space pertaining to this case.

In this case, the affine Dynkin diagram is of type $\tilde{B}_m$ with vertices $0$, $1$, \dots, $m$, and $\sigma$ acts by  exchanging $0$ and $1$, and fixing all other vertices. Conjugation by $\tau$ also exchanges $0$ and $1$ and fixes the other vertices, so that $\tau\sigma = \id$.

The level $K$ in this example is given by $\tSS \setminus \{m\}$. It corresponds to a non-special rational maximal parahoric subgroup.

Let $\mu =\omega_1^\vee$. The set $\EOcox$ consists of $\tau$, $s_m \tau$, \dots, $s_ms_{m-1}\cdots s_2\tau$, $s_ms_{m-1}\cdots s_2s_1\tau$, $s_ms_{m-1}\cdots s_2s_0\tau$. The corresponding $\s$-supports are just the usual supports ($\emptyset$, $\{ m\}$, \dots, $\{ 2, \dots, m\}$, $\{ 1, 2, \dots, m\}$, $\{ 0, 2, \dots, m\}$).

The sets $\Sigma_w$ are the sets $\{i\}$ for $i=0, \dots, m$, and the set $\{0, 1\}$ (the latter case for $m=6$ is shown in the figure; $w=s_6s_5s_4s_3s_2\tau$,  $\Sigma_w = \{ 0, 1 \}$ (the two vertices marked with gray circles), $\supps(w) = \{ 2, 3, 4, 5, 6 \}$ (the vertices marked with black circles).

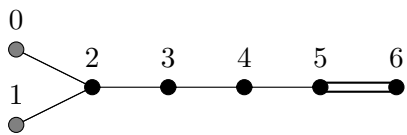
\begin{figure}[h]
\centering
\begin{tikzpicture}
\draw[fill=black] 
(3,0)                         
circle [radius=.1] node [label=$6$] {};

\draw[fill=black] 
(2,0)                         
      circle [radius=.1] node [label=$5$] {} --
(1,0)                         
      circle [radius=.1] node [label=$4$] {} --
(0,0)                         
circle [radius=.1] node [label=$3$] {} --
(-1,0) 
circle [radius=.1] node [label=$2$] {}
;

\draw[fill=black] 
(-1,0)  --
(-2,0.5);
\draw[fill=black] 
(-1,0)  --
(-2,-0.5)
; 

\draw[fill=gray]  (-2, 0.5)    circle [radius=.1] node [label=$0$] {}; 
\draw[fill=gray]  (-2, -0.5)    circle [radius=.1] node [label=$1$] {};


\draw[thick] (2, 0.06) -- +(1,0);
    \draw[thick] (2, -0.06) -- +(1,0);

\end{tikzpicture}
\caption{Example 3.2 for $m=6$, $w=s_6 s_5 s_4 s_3 s_2\tau$, $\Sigma_w= \{0, 1\}$.}
\end{figure}
\end{example}

Now we return to the general case.

\emph{The choice of representative of the $\sigma$-conjugacy class.}
The most interesting among the spaces $X(\mu, b)_K$ is the one where the $\s$-conjugacy class of $b$ is basic.

In this case, we choose as a representative (an arbitrary lift of) the unique length $0$ element $\tau\in \tW$ such that $\varepsilon^\mu\in W_a\tau$. Equivalently, $\tau$ is the unique length $0$ element such that $X(\mu, \tau)_K$ is non-empty. This is the choice made in \cite{GH}, as well, and the choice suggested as a reasonable one in \cite{CV} Remark 2.1. (As explained there, the $\JJ$-stratification varies significantly if a different representative, and correspondingly, a different group $\JJ$, is used. But note that also the treatment in \cite{GH}, e.g.~Prop.~2.2.1 there, relies on the use of the particular representative $\tau$.)

\begin{lemma}
For $w\in\tW$, any two lifts of $w$ to $N(\brF)\subset G(\brF)$ are $T(\brF)_1$-$\sigma$-conjugate and hence the $\JJ$-stratification is independent of the choice of lift.
\end{lemma}

\begin{proof}
We use Lang's theorem for the parahoric subgroup $T(\brF)_1$ of $T(\brF)$ (equipped with a ``twisted'' Frobenius $\mathop{\rm Int}(w)\circ \s$, where $w$ is a fixed lift), see~\cite{Greenberg-II}, Prop.~3. We obtain that the map $T(\brF)_1\to T(\brF)_1$, $g\mapsto g^{-1} w\sigma(g)w^{-1}$ is surjective. But an element $t$ is in the image if and only if $tw$ is $\s$-conjugate under $T(\brF)_1$ to $w$.


As pointed out in \cite{CV} Remark 2.1, elements $b_1, b_2\in G(\brF)$ which are $\s$-conjugate under $\brK$ give rise to the same $\JJ$-stratification, in the following sense: for $b_2 = k^{-1} b_1 \s(k)$, the isomorphism $X(\mu, b_2)_K \to X(\mu, b_1)_K$, $g\mapsto kg$ identifies $\JJ_{b_2}$-strata with $\JJ_{b_1}$-strata. This is clear using the isomorphism $\JJ_{b_2}\to \JJ_{b_1}$, $j\mapsto kjk^{-1}$.
\end{proof}

\subsection{The \BTstrat}\label{sec:bt-stratification}

As above, let $\pi\colon \Flag \to \Flag_K$ be the projection. By definition of the notion \emph{of Coxeter type} in \cite{GH} (right before Theorem 5.1.2) we obtain a decomposition of $X(\mu, \tau)_K$ as a disjoint union
\[
X(\mu, \tau)_K = \bigsqcup_{w\in \EOcox} \pi(X_w(\tau))
\]
Further, each $\pi(X_w(\tau))$ can be decomposed as
\[
\pi(X_w(\tau)) = \bigsqcup_{i\in \JJ/(\JJ\cap \brP_{\tSS\setminus \Sigma_w})} iY(w),
\]
with
\[
Y(w) = \pi(\{ g\in \brP_{\supps(w)}/\brI;\ g^{-1}\tau\s(g) \in \brI w\brI \}).
\]
The resulting stratification of $X(\mu, \tau)_K$,
\[
X(\mu, \tau)_K = \bigsqcup_{w\in\EOcox}\quad \bigsqcup_{i\in \JJ/(\JJ\cap \brP_{\tSS\setminus \Sigma_w})} iY(w)
\]
is called the \emph{Bruhat-Tits stratification} of $X(\mu, \tau)_K$.
Here $\brP_{\tSS\setminus \Sigma_w}$ denotes the parahoric subgroup of $G(\brF)$ attached to $\tSS\setminus \Sigma_w \subset \tSS$. Since we identify $\JJ(\brF) = G(\brF)$, we can intersect it with $\JJ = \JJ(F)$.

The term Bruhat-Tits stratification is used because the index set can be described in terms of the Bruhat-Tits building for $\JJ$. In the most common case where $\Sigma_w$ is a single $\tau\sigma$-orbit, $\JJ\cap \brP_{\tSS\setminus \Sigma_w}$ is the stabilizer of the corresponding vertex of the base alcove in the rational building of $\JJ$ and hence $\JJ/(\JJ\cap \brP_{\tSS\setminus \Sigma_w})$ can be identified with the set of all vertices of ``type $\Sigma_w$''. The Bruhat-Tits stratification is a stratification in the strict sense that the closure of each stratum is a union of strata. In fact, the closure relations between strata can be described explicitly in terms of the building, see~\cite{GH} Section 7.

Note that $Y(w)$ is (isomorphic to) a classical Deligne-Lusztig variety. Indeed, the projection $\pi$ induces an isomorphism
\[
\{ g\in \brP_{\supps(w)}/\brI;\ g^{-1}\tau\s(g) \in \brI w\brI \} \cong Y(w)
\]
since $w$ is a $\s$-Coxeter element (with respect to $\supps(w)$). The quotient $\brP_{\supps(w)}/\brI$ can be identified with the finite-dimensional flag variety for the maximal reductive quotient $G_w$ of the reduction of the parahoric subgroup $\brP_{\supps(w)}$ (with respect to its Borel subgroup induced by $\brI$). Equipping $G_w$ with the Frobenius $x\mapsto \tau\sigma(x)\tau^{-1}$, we find that $Y(w)$ is defined by the usual Deligne-Lusztig condition for the Weyl group element $w\tau^{-1}\in W_{\supps(w)}$. The Dynkin diagram of $G_w$ is given by the part of the affine Dynkin diagram with vertices $\supps(w)$, and the Frobenius action on this Dynkin diagram is given by $\tau\sigma$. The finite Weyl group of $G_w$ can be identified with $W_{\supps(w)}$ and the element $w\t^{-1}$ is a twisted Coxeter element in this group. For instance, in the situation of Example~\ref{example1}, we obtain classical Deligne-Lusztig varieties in a special unitary group.

\begin{theorem}\label{main-thm}
Let $(G, \mu, K)$ be of Coxeter type, and let $\tau\in \tW$ be the length $0$ element representing the unique basic $\sigma$-conjugacy class such that $X(\mu, \tau)_K\ne\emptyset$. Then the Bruhat-Tits stratification on $X(\mu, \tau)_K$ coincides with the $\JJ$-stratification.
\end{theorem}

The theorem will be proved in the following sections. In particular we see that in these special cases, the $\JJ$-stratification is a stratification in the strict sense and has many nice properties. In the case of $G=GSp_4$ with hyperspecial level structure, and in the case of Example 3.1, the Theorem was proved in \cite{CV}, 4.2.1 and 4.2.2, resp., by direct considerations in terms of Dieudonn\'e modules.

Note that in the case where $K$ is not assumed to be hyperspecial, $X(\mu, b)_K$ will in general be a disjoint union of several affine Deligne-Lusztig varieties. Of course, as a corollary we see that the statement of the theorem also holds for these usual affine Deligne-Lusztig varieties contained in $X(\mu, b)_K$.

\begin{remark}\label{non-basic-case}
If $b$ is not basic, then under the above assumptions, $X(\mu, b)_K$, if non-empty, is $0$-dimensional (see~\cite{GH} Thm.~5.2.1, cf.~also \cite{GHN2}, Theorem 2.3). The strata of the Bruhat-Tits stratification are 
the individual points of this scheme. More precisely, $X(\mu, b)_K$ is a union of the form $\bigcup_i \pi(X_{w_i}(b))$ for certain $\sigma$-straight elements $w_i\in\Adm(\mu)$, where $\pi\colon \Flag\to \Flag_K$ denotes the projection. The group $\JJ=\JJ_{b}(F)$ acts transitively on each $X_{w_i}(b)$, and hence on its image $\pi(X_{w_i}(b))$.

In particular, if the index set for the above union has only one element $w$, then $\JJ$ acts transitively. In this case it is immediate that the strata of the $\JJ$-stratification are the individual single points of $X(\mu, b)_K$, if we choose the element $w$ as the representative of the $\sigma$-conjugacy class $[b]$.

However, in certain cases, the index set for the above union has more than one element. In this case, the $\JJ$-action on $X(\mu, b)_K$ is not transitive (contrarily to what is claimed in \cite{GH} Thm.~B (2)). See \cite{GH} 6.4 and 6.6 for the two cases where this occurs. As one checks from the data given in loc.~cit., the index set then consists of exactly two elements $w_1$, $w_2$ with $w_2 = \tau w_1\tau^{-1}$, and conjugation by $\tau$ stabilizes the subset $K\subset \tSS$, and hence conjugation by $\tau$ stabilizes the parahoric subgroup $\brK$. We choose one of the $\sigma$-straight elements $w_1$, $w_2$ as the representative of the $\sigma$-conjugacy class of $b$. Since conjugation by $\tau$ stabilizes $\brK$, it is again easy to check that the $\JJ$-stratification obtained in this situation has as its strata the individual points of $X(\mu, b)_K$.
\end{remark}

\subsection{The BT-stratification is finer than the $\JJ$-stratification}

We start by proving the easier direction of Theorem~\ref{main-thm}. Recall that $\JJ$ is an inner form of $G$ since $\tau$ is basic. We fix an identification $G(\brF) = \JJ(\brF)$. We have the action of $\sigma$ on $G(\brF)$ with fixed points $G(F)$, and the twisted action by $\mathop{\rm Int}(\tau)\circ\sigma$, $x\mapsto \tau \sigma(x)\tau^{-1}$, with fixed points $\JJ$.

Given $w\in \EOcox$, we have $P:=\supps(w)$ and the corresponding parahoric subgroup $\brP$.

We need to show that for all $j\in \JJ$, the value $\inv_K(j, -)$ is constant on each BT stratum $j' Y(w)$. We may pass to Iwahori level and will show the stronger statement that $\inv(j, -)$ is constant on $j' Y(w) \subset j'\brP/\brI$.

Since $\inv(j, j'h) = \inv((j')^{-1}j, h)$, we may assume that $j' = 1$. Similarly, given $j\in \JJ$, there exists $\tilde{\jmath}\in \JJ_1$ such that $j\fka = \tilde{\jmath}\fka$; then conjugation by $j^{-1}\tilde{\jmath}$ preserves $\brI$ and so
$\brI j^{-1} h\brI = j^{-1} \tilde{\jmath} \brI \tilde{\jmath}^{-1}h \brI$. This allows us to restrict ourselves to the case that $j\in \JJ_1$.

Altogether we see that it is enough to show that for every alcove $\fkj$ in $\mathcal B(\JJ, F)$, $\delta(\fkj, \fkh)$ is independent of $\fkh \in Y(w)$, where we now consider the latter set as a set of alcoves in $\mathcal B(\JJ, \brF)=\mathcal B(G, \brF)$.

Now let $\fkg$ be the gate from $\fkj$ to the residue $\brP/\brI$ in the sense of Section~\ref{sec:gate}. It is an alcove inside the rational building $\mathcal B(\JJ, F)$ (Remark~\ref{rmk-gate-rational}), and hence Lusztig's result about Deligne-Lusztig varieties attached to twisted Coxeter elements, Proposition~\ref{prop-coxeter-dlv}, give us that $\delta(\fkg, \fkh) = w_{0, P}$, the longest element of $W_P$, for all $\fkh\in Y(w)$. We then have that
\[
\delta(\fkj, \fkh) = \delta(\fkj, \fkg)\delta(\fkg, \fkh) =  \delta(\fkj, \fkg) w_{0, P}
\]
is independent of $\fkh\in Y(w)$.

\subsection{Acute cones and extension of galleries}

We recall some definitions and results from \cite{HN} Section 5. The notion of acute cone and the discussion until and including Definition~\ref{def-acute-cone} applies to a fixed apartment $A$ (and the corresponding finite Weyl group $W_0$, set of positive roots $\Phi^+$, and the affine root hyperplanes $H_{\alpha, k}$, \dots) in some affine building. We will eventually apply it to the rational building of the group $\JJ$.

Let $w\in W_0$. An affine root hyperplane $H= H_{\alpha, k} (=H_{-\alpha, -k})$ defines two half-spaces in $A$. To distinguish between them, let us assume that $\alpha \in w(\Phi^+)$. Then we let $H^{w+} = \{ v\in A;\ \langle \alpha, v\rangle > k\}$ and call this the $w$-positive half-space. We call the other half-space the $w$-negative half-space and denote it by $H^{w-}$.

\begin{definition} (\cite{HN} Def.~5.2)
Let $w\in W_0$. We say that a non-stuttering gallery $\Gamma = (\fkb_0, \dots, \fkb_\ell)$ in $A$ \emph{goes in the $w$-direction}, if for all $i = 1, \dots, \ell$, denoting by $H_i$ the wall separating $\fkb_{i-1}$ and $\fkb_i$, the alcove $\fkb_i$ lies in $H_i^{w+}$.
\end{definition}

\begin{proposition}{\rm (\cite{HN} Lemma 5.3, Cor.~5.6)}\label{prop-gall-w-direction}
\begin{enumerate}
\item
Let $\Gamma$ be a gallery going in the $w$-direction ($w\in W_0$). Then $\Gamma$ is a minimal gallery, and every minimal gallery between the two end alcoves of $\Gamma$ goes in the $w$-direction.
\item
Let $\Gamma$ be a minimal gallery in $A$. Then there exists $w\in W_0$ such that $\Gamma$ goes in the $w$-direction.
\end{enumerate}
\end{proposition}

\begin{corollary}\label{cor-concat-gall}
Let $w\in W_0$, and let $\Gamma_1$, $\Gamma_2$ be galleries which both go in the $w$-direction, and such that the final alcove of $\Gamma_1$ is the first alcove of $\Gamma_2$. Then the concatenation of $\Gamma_1$ and $\Gamma_2$ is minimal.
\end{corollary}

\begin{proof}
By the proposition, it is enough so show that the concatenation of $\Gamma_1$ and $\Gamma_2$ again goes in the $w$-direction, and this is clear by definition.
\end{proof}

\begin{definition} (\cite{HN} Def.~5.4)\label{def-acute-cone}
Let $\fkb$ be an alcove in $A$, and let $w\in W_0$.

The \emph{acute cone} $C(\fkb, w)$ is the set of all alcoves $\fkb'$ in $A$ such that there exists a gallery from $\fkb$ to $\fkb'$ going in the $w$-direction.
\end{definition}


Now we will apply these general notions to the situation at hand. Below, when we say that a translation element $\varepsilon^\lambda$ in $\breve{\mathscr A}$ lies far from all $\brF$-walls, we mean that it lies far from all \emph{finite} root hyperplanes, in other words: $|\langle\alpha, \lambda \rangle| \gg 0$ for all $\alpha\in\Phi$. Similarly, we have the notion of an alcove lying far from all $\brF$-walls. (Another way to say this would be to say that the alcove lies in a ``very shrunken'' Weyl chamber.)

\begin{lemma}\label{lemma-translations-far-from-walls}
Suppose that the affine Dynkin diagram of $G$ is connected,
and that $\JJ$ has $F$-rank $> 0$.

There exists a $\JJ$-rational translation of the base alcove which is far from all $\brF$-walls.
\end{lemma}

\begin{proof}
According to our setup, the $\JJ$-rational apartment contains the barycenter $v$ of the $\brF$-base alcove. Since $\mathop{\rm rk} _F \JJ >0$, the origin $v'$ of the $\JJ$-rational standard apartment is the barycenter of some proper facet of the $\brF$-base alcove. It follows from a simple computation (relying on the assumption that the affine Dynkin diagram is connected, i.e., that alcoves and their faces are simplices, not just ``polysimplices'') that for all roots $\alpha$, $\langle \alpha, v\rangle \ne \langle \alpha, v'\rangle$.

Therefore for $N$ sufficiently divisible, $v' + N(v-v')$ is a rational translation element which is far from all $\brF$-walls.
\end{proof}

\begin{lemma}\label{lemma-extend-gallery}
Suppose that the affine Dynkin diagram of $G$ is connected, and that $\mathop{\rm rk}_F\JJ > 0$.
Let $\Gamma = (\fkb_0, \fkb_1, \dots, \fkb_\ell)$ be a minimal gallery in $\breve{\mathcal B}$ such that $\fkb_\ell$ and $\delta(\fkb_0, \fkb_\ell)$ are fixed by $\sigma_\JJ$.
Then there exists a minimal gallery $\Gamma'$ in $\breve{\mathscr A}$ which starts in $\fkb_\ell$, ends in an alcove which is fixed by $\sigma_\JJ$ and which is far away from all $\brF$-walls, and such that the concatenation of $\Gamma$ and $\Gamma'$ is a minimal gallery.
\end{lemma}

\begin{proof}
Applying a suitable element of $\JJ$ to the whole situation, we may assume that $\fkb_\ell = \fka$. We have the canonical retraction $\rho$ from the building to the standard apartment with respect to $\fka$ (cf.~\cite{AB} Def.~4.38). If $\mathscr A'$ is any apartment containing $\fka$, it induces an isomorphism between $\mathscr A'$ and the standard apartment. Since $\Gamma$ is minimal, there exists an apartment $\mathscr A'$ containing $\Gamma$, and $\rho$ thus maps $\Gamma$ to a minimal gallery in the standard apartment. Therefore it is enough to extend the image of $\Gamma$ under $\rho$ inside the standard apartment. In other words, we may assume that $\Gamma$ is contained in the standard apartment and ends in $\fka$.

Since $\delta(\fkb_0, \fkb_\ell)$ and $\fkb_\ell=\fka$ are fixed by $\sigma_\JJ$, $\fkb_0$ is also fixed by $\sigma_\JJ$ (not necessarily point-wise). Here we use that we have reduced, in the first step, to the case that $\Gamma$ is contained in a $\JJ$-rational apartment.

We apply Prop.~\ref{prop-gall-w-direction} to the building $\mathcal B(\JJ, F)$ (i.e., we consider $\delta(\fkb_0, \fkb_\ell)$ as an element of the Coxeter group $W_a^\JJ$, and accordingly get a minimal gallery in the standard apartment of $\mathcal B(\JJ, F)$): There exists $w$ in the corresponding finite Weyl group such that $\Gamma$ goes in the $w$-direction.

All the alcoves in the acute cone $C(\fka, w)$ (understood inside the standard apartment of $\mathcal B(\JJ, F)$) can, by definition, be reached from $\fka$ by a gallery in the $w$-direction, and concatenating such a gallery with $\Gamma$ gives us another minimal gallery, by Cor.~\ref{cor-concat-gall} (equivalently, a reduced word in $W_a^\JJ$ with initial piece $\delta(\fkb_0, \fkb_\ell)$). So it is enough to show that $C(\fka, w)$ contains alcoves arbitrarily far from all $\brF$-walls. Since $C(\fka, w)$ contains a translate of the finite Weyl chamber corresponding to $w$ (cf.~\cite{HN} Prop.~5.5), and the finite Weyl group of the $\JJ$-rational standard apartment acts transitively on the set of $\JJ$-rational finite Weyl chambers, the desired statement follows from Lemma~\ref{lemma-translations-far-from-walls}.
\end{proof}

\begin{example}
Clearly, the hypothesis that the affine Dynkin diagram of $G$ is connected is required in the above lemmas: Consider $\widetilde{A}_1 \times \widetilde{A}_1$ with $\mathbb Z/2\mathbb Z$ acting on the first factor by the non-trivial diagram automorphism, and on the second factor as the identity. Then the rational apartment is parallel to one of the walls.
\end{example}

\begin{lemma}\label{lem-far-from-walls}
Let $u, u' \in \tW$ such that $uW_K\ne u'W_K$, and let $w\in \tW$ such that $w\fka$ is sufficiently far away from the walls (depending on $u$, $u'$). Then $W_KwuW_K \ne W_Kwu'W_K$.
\end{lemma}

\begin{proof}
For $K= \tSS\setminus\{v\}$, we define the set of $K$-walls as the set of root hyperplanes of the (finite) root system obtained by changing the origin to $v$, i.e., the set of affine root hyperplanes passing through the vertex of the base alcove $\fka$ corresponding to $K$ (this construction also appeared in Section~\ref{sec:finiteness-property}). E.g., if $v= s_0$, then the $K$-walls are precisely the finite root hyperplanes. The $K$-walls are the root hyperplanes of a root system with Weyl group $W_K$. The chambers of this root system will be called the $K$-chambers.

Since the origin we fixed in $\breve{\mathscr A}$ is by assumption a special point, every $K$-wall is parallel to some finite root hyperplane. It follows that an alcove which is far from all finite root hyperplanes is also far from all $K$-walls.

If $w$ is far away from the $K$-walls, the finite sets $wuW_K$ and $wu'W_K$ are contained in the same $K$-chamber. Therefore to show the claim it is enough to ensure that $wuW_K \ne wu'W_K$ which follows directly from the assumption $uW_K \ne u'W_K$.
\end{proof}

\subsection{Subexpressions}

Given a Coxeter system $(W, \mathbb S)$ and $w\in W$, we say that a reduced expression $w=f_1, \dots f_{\ell(w)}$ (with the $f_i$ being simple affine reflections) has a tuple $(t_1, \dots, t_r)\in \BS^r$ as a subexpression, if the tuple $(t_1, \dots, t_r)$ can be obtained from $(f_1, \dots, f_{\ell(w)})$ by omitting some of the entries (but without changing their order).

\begin{lemma}\label{lem-subexpr}
Let $(W, \BS)$ be a Coxeter system. Let $t_0, \dots, t_r\in \BS$ be pairwise different elements such that $t_i t_{i-1} \ne t_{i-1} t_i$ for all $i=1, \dots, r$. Let $w\in W$ such that $t_0\cdots t_r \le w$.
\begin{enumerate}
\item
Every reduced expression for $w$ has $(t_0, \dots, t_r)$ as a subexpression.
\item
Let $s\in \BS$, $s\ne t_0$. Then $t_0\cdots t_r \le sw$. In particular $t_r$ is contained in the support of $sw$.
\end{enumerate}
\end{lemma}

\begin{proof}
We start by proving the following special case of part (1): The element $t_0\cdots t_r \in W$ has a unique reduced expression (namely the specified one). This follows immediately from the result of Tits that any two reduced expressions for an element of a Coxeter group can be obtained from each other by a sequence of ``elementary homotopies'' (see~\cite{AB} Prop.~3.24), i.e., by applying the braid relations for a pair $s, s'\in \BS$ of simple reflections (including the case where $s$, $s'$ commute).

To prove (1) in the general case, assume $t_0\cdots t_r \le w$ for some $w$. The subword property of the Bruhat order (see e.g., \cite{Bjorner-Brenti} Thm.~2.2.2) states that starting from any reduced expression for $w$ we can obtain some reduce expression for $t_0\cdots t_r$ by omitting certain factors. By the special case proved already, this implies part (1).

Now part (2) follows easily: If $\ell(sw) = \ell(w)+1$, then there is nothing to prove. But if $\ell(sw) = \ell(w)-1$, then $w$ has a reduced expression starting with $s$, and $sw$ can then be obtained by omitting the first $s$ in that reduced expression. Since $s\ne t_0$, the assertion follows from (1).
\end{proof}

\subsection{The $\JJ$-stratification is finer than the BT-stratification}

We will now show that given EO types $w, w'$ and $j'\in \JJ$ with $Y(w)\ne j'Y(w')$, there exists $j\in \JJ$ such that $\inv_K(j, -)$ has different values on $Y(w)$ and $j'Y(w')$. (We know already that these values are constant on $Y(w)$ and $j'Y(w')$, resp.) This clearly implies that any two strata can be separated by the relative position with respect to some element of $\JJ$. We may and will also assume without loss of generality that $P=P'$ or $P' \not\subseteq P$.

It is easy to see that for $j' \in \JJ\setminus \JJ_1$ and arbitrary $j$, the values of $\inv_K(j, -)$ on $Y(w)$ and $j' Y(w')$ lie in different $W_a$-cosets within $\tW$ (i.e., the length $0$ component is different). Therefore we will assume from now on that $j'\in \JJ_1$, and we will find $j\in \JJ_1$ separating the two strata. This means that we may rephrase the problem by saying that we search an alcove $\fkj$ in the $\JJ$-rational building such that the $K$-Weyl distance $\delta_K(\fkj, -)$ separates the strata $Y(w)$ and $j' Y(w')$.

\begin{remark}
\begin{enumerate}
\item
It is easy to separate different strata in the full affine flag variety, using the Iwahori-Weyl distance. To ensure that the relative positions remain different after passing to $\brK$-double cosets, we will use Lemma~\ref{lem-far-from-walls} to reduce to a problem of comparison of left cosets rather than double cosets. This can eventually be handled using Lemma~\ref{lem-subexpr}.
\item
The elements $j\in \JJ$ produced by the proof below which separate different Bruhat-Tits strata will typically lie far away from the origin (when considered as alcoves). It seems difficult to give an explicit bound on this distance. It is easy to check in examples that the identity element $j=1$ does not necessarily separate all Bruhat-Tits strata of the form $Y(w)$, $Y(w')$, $w\ne w'$.
\end{enumerate}
\end{remark}

Denote by $P$ and $P'$, resp., the $\sigma$-supports of $w$ and $w'$, and by $\brP$, $\brP'$ the corresponding parahoric subgroups, so $Y(w)\subseteq \brP/\brI$, $j' Y(w')\subseteq j' \brP'/\brI$.

We consider $\brP/\brI$ and $j'\brP'/\brI$ as sets of alcoves in $\mathcal B(G, \brF)$. Both sets are residues in the sense of Section~\ref{sec:gate} and we apply~Prop.~\ref{AB-5.37} with $\CR = \brP/\brI$, $\CR' = j'\brP'/\brI$. Denote by $w_1$ the ``minimal distance'' between $\CR$ and $\CR'$, i.e., $w_1 = \min(\delta(\CR, \CR'))\in W_a^\JJ$, and let $\CR_1 = \proj_{\CR}(\CR')$, $\CR'_1  =\proj_{\CR'}(\CR)$. The proposition gives us a bijection between $\CR_1$ and $\CR'_1$ such that the relative position between any two alcoves corresponding to each other is $w_1$.

Let $\fkh\in Y(w)$, $\fkh'\in j'Y(w')$. Fix any $\JJ$-rational alcove $\fkb$ in $\mathcal R_1$, and let $\fkb'\in \CR'_1$ be the element corresponding to $\fkb$ under the above bijection. By Rmk.~\ref{rmk-gate-rational}, $\fkb'$ lies in the rational building. We obtain (using~Prop.~\ref{prop-coxeter-dlv}), denoting by $w_{0, P'}$ the longest element in $W_{P'}$,
\[
\delta(\fkb, \fkh') = \delta(\fkb, \fkb')\delta(\fkb', \fkh') = w_1 w_{0, P'}
\]
and $\ell(w_1 w_{0, P'}) = \ell(w_1) + \ell(w_{0, P'})$, i.e., choosing reduced expressions for $w_1$ and $w_{0, P'}$ we obtain a minimal gallery starting at the alcove $\fkh'$ with end point the alcove $\fkb$. We extend this minimal gallery to a minimal gallery from  $\fkh'$ to an alcove $\fkj\in\mathcal B(\JJ, F)$ which is far from the $K$-walls.  Here the element is sufficiently far away from the walls as soon as we can apply Lemma~\ref{lem-far-from-walls} for any $u\in W_P$ and $u'\in W_Pw_1w_{0, P'}$. This imposes only finitely many conditions and can therefore be achieved by Lemma~\ref{lemma-extend-gallery} (the affine Dynkin diagram is connected, and the $F$-rank of $\JJ$ is $>0$ in all Coxeter type situations).

If we replace the part from $\fkb$ to $\fkj$ in our gallery by any other minimal gallery from $\fkb$ to $\fkj$, the resulting gallery will still be minimal. Therefore we may assume that the gate $\fkg$ from $\fkj$ to $\CR$ is one of the alcoves of the gallery. This is a $\JJ$-rational alcove, and writing $w_3 =\delta(\fkj, \fkg)$ we have $\delta(\fkj, \fkh)=\delta(\fkj, \fkg)\delta(\fkg, \fkh)= w_3 w_{0, P}$.

Let $w_2:=\delta(\fkg, \fkb)$ be the relative position between $\fkg$ and $\fkb$.  Note that $w_2\in W_P$ because $\CR$ is convex and thus contains every minimal gallery between any two of its elements.

Summarizing, we have
\begin{align*}
& \delta(\fkj, \fkh) = \delta(\fkj, \fkg)\delta(\fkg, \fkh)= w_3 w_{0, P},\\
& \delta(\fkj, \fkh') = \delta(\fkj, \fkg)\delta(\fkg, \fkb)\delta(\fkb, \fkb')\delta(\fkb', \fkh')  = w_3 w_2 w_1 w_{0, P'}.
\end{align*}
The inequality $\delta_K(\fkj, \fkh) \ne \delta_K(\fkj, \fkh')$ which we would like to show thus amounts to the inequality
\[
W_K w_3 w_{0,P} W_K \ne W_K w_3 w_2 w_1 w_{0, P'} W_K.
\]
By Lemma~\ref{lem-far-from-walls} we see that it is enough to show that $w_{0, P}W_K \ne w_2 w_1 w_{0, P'} W_K$, or in other words that
\[
w_{0,P} w_2 w_1 w_{0, P'} \not\in W_K.
\]
Another equivalent reformulation of this condition is that $s_v \in \supp(w_{0,P} w_2 w_1 w_{0, P'})$ (where $\{v \} = \tSS\setminus K$ as above). To prove this, we distinguish cases.

\subsubsection{The case $j' Y(w') = Y(w')$}

In this case, we may assume that $j' = 1$, and have that $w_1=1$. Since $Y(w)\ne Y(w')$, we have $w\ne w'$, $P\ne P'$, and by the reduction step in the very beginning, we then have $P' \not\subseteq P$.

There exist simple reflections $s_v = t_r, t_{r-1}, \dots, t_0 \in W_{P'}$ such that the $t_i$ are pairwise different, $t_i\in W_P$ if and only if $i > 0$, and such that $t_i t_{i-1} \ne t_{i-1} t_i$ for all $i$. In fact, to find such elements, choose an arbitrary element of $P' \setminus P$, and connect it in the affine Dynkin diagram to $v$. Now define the $t_i$ starting with $t_r=s_v$ and following the chosen path until it leaves $P$ for the first time.

Then $t_0 t_1 \cdots t_r \le w_{0, P'}$. Applying Lemma~\ref{lem-subexpr} multiple times and using that $t_0\not\in W_P$, we see that for every element $u\in W_P$, the product $uw_{0, P'}$ also has $(t_0, t_1, \dots, t_r)$ as a subexpression, and in particular its support contains $t_r=s_v$. We apply this to $u = w_{0,P}w_2$ and obtain that $w_{0, P} w_2 w_1 w_{0, P'} = uw_{0, P'}\not\in W_K$.

\subsubsection{The case $j'Y(w')\ne Y(w')$}

We use the notation introduced in Section~\ref{sec:bt-stratification}. Specifically, let $\Sigma=\Sigma_w$ and $\Sigma'=\Sigma_{w'}$ denote the $\tau\sigma$-stable subsets of $\tSS$ consisting of the vertices in the affine Dynkin diagram adjacent to $P$ but $\not\in P$ (resp., $P'$).

The assumption that $j' Y(w') \ne Y(w')$ in $\Flag_K$ means that $j' \not\in \brP_{\tSS\setminus \Sigma'}$. In other words, $\supp(w_1)\cap \Sigma' \ne\emptyset$.

\medskip
\emph{Claim.} There exist simple reflections $s_v = t_r, \dots, t_1 \in W_{P'}$ and $t_0\le w_1$, $t_0\not\in W_P$ such that the elements $t_r, \dots t_0$ satisfy the assumptions of Lemma~\ref{lem-subexpr}.

\medskip
Let us first show that with elements as in the claim, we can finish the proof. Since $\ell(w_1 w_{0, P'}) = \ell(w_1) + \ell(w_{0, P'})$, $t_0\le w_1$ and $t_1 \cdots t_r\le w_{0, P'}$ imply that $t_0 t_1\cdots t_r\le w_1 w_{0, P'}$. We also have that $t_0\not\in\supp(w_{0, P} w_2)$ since the latter set is a subset of $P$. By Lemma~\ref{lem-subexpr} it follows that $s_v = t_r \in \supp(w_{0, P}w_2w_1w_{0, P'})$, as desired.

It remains to prove the claim. We have already seen that $\Sigma'\cap\supp(w_1)\ne \emptyset$, and if $\Sigma'\cap P = \emptyset$, we can choose any element in $\Sigma'\cap \supp(w_1)$ as $t_0$, and then define $t_1, \dots, t_r\in W_{P'}$ by ``connecting'' the vertex corresponding to $t_0$ with $v$ in the affine Dynkin diagram.

It remains to discuss the case $\Sigma' \cap P \ne \emptyset$. Since we have also (in the very beginning of the proof) reduced to the case $P' \not\subseteq P$, this can only happen in the case where the sets $\supps(u)$, $u\in \EOcox$, are \emph{not} linearly ordered by inclusion. More precisely, we must be in the case of Example~\ref{example2} (the affine Dynkin diagram is of type $\tilde{B}$, and $\sigma$ and $\tau$ both act by exchanging $0$ and $1$ and fixing all other vertices), and we have either $\Sigma= \{0\}$, $\Sigma' = \{1\}$, or $\Sigma= \{1\}$, $\Sigma' = \{0\}$.

We then want to choose as $t_0$ the simple reflection associated with the element of $\Sigma$. Looking at the affine Dynkin diagram in this case, we see immediately that we then find $t_1, \dots, t_r$ as required in the claim. To justify this choice, we need to show that $t_0 \le w_1$. We know that $w_1$ is the minimal length element in $W_Pw_1 W_{P'}$, and that it is $\ne 1$. Therefore, we cannot have $\supp(w_1)\subseteq P$. Since $\tSS \setminus P = \Sigma = \{t_0\}$ consists of only one element in the case at hand, it follows that indeed $t_0\in \supp(w_1)$.





\end{document}